\documentclass[a4paper,12pt]{amsart}
\usepackage{amssymb}
\usepackage{amsmath, amsthm, amscd, amsfonts, amssymb, graphicx, color}
\usepackage{amsmath, amsthm, amscd, amsfonts, amssymb, graphicx, color}

\newtheorem{theorem}{Theorem}[section]
\newtheorem{lemma}[theorem]{Lemma}
\newtheorem{proposition}[theorem]{Proposition}
\newtheorem{corollary}[theorem]{Corollary}
\theoremstyle{definition}

\newtheorem{example}[theorem]{Example}

\theoremstyle{remark}
\newtheorem{remark}[theorem]{Remark}
\numberwithin{equation}{section}
\title { Izbačeno }

\begin{document}

\author{Stefan Ivkovi\'{c} }

\vspace{25pt}
	
\title{The generalized spectra in $ C^{*}$-algebras of operators over $C^{*}$-algebras}
       
\maketitle

\begin{abstract}
		In this paper we consider shift operators, self-adjoint, unitary and normal operators on the standard module over a unital $C^{*}$-algebra $\mathcal{A}$. We define various generalized spectra in $\mathcal{A}$ of these operators, give description of such spectra of these operators and investigate their further properties.
\end{abstract} 

\vspace{45pt}
	
\begin{flushleft}
The Mathematical Institute of the Serbian Academy of Sciences and Arts, \\
p.p. 367, Kneza Mihaila 36, 11000 Beograd, Serbia,\\
Tel.: +381-69-774237 \\
email: stefan.iv10@outlook.com  	
\end{flushleft} 


\begin{flushleft}
\end{flushleft}


\section{Introduction}
Hilbert $C^{*}$-modules are natural generalizations of Hilbert spaces when the field of scalars is replaced by an arbitrary $C^{*}$-algebra. In this paper we let $\mathcal{A}$ be a unital $C^{*}$-algebra and $H_{\mathcal{A}} $ denote the standard module over $\mathcal{A},$ that is $  H_{\mathcal{A}}=l_{2}(\mathcal{A}).$ Moreover, we let $B^{a}(H_{\mathcal{A}}) $ be the $C^{*}$-algebra of all $\mathcal{A}$-linear, bounded, adjointable operators on $H_{\mathcal{A}} .$

Our starting question is the following: If $\mathcal{A}$ is a $C^{*}$-algebra, then for $\alpha \in \mathcal{A}$ could we define in a suitable way the operator $\alpha I $ on $H_{\mathcal{A}} $ and the generalized spectra in $\mathcal{A}$ of operators in $B^{a}(H_{\mathcal{A}}) $ by setting for every $F \in B^{a}(H_{\mathcal{A}}) 
$ $\sigma^{\mathcal{A}}(F)=\lbrace \alpha \in \mathcal{A} \mid F-\alpha I  $  is not invertible in$B^{a}(H_{\mathcal{A}}) \rbrace ?$\\
For $a \in \mathcal{A} $ we may let $\alpha I $ be the operator on $H_{\mathcal{A}} $ given by $\alpha I (x_{1},x_{2}, \cdots)=(\alpha x_{1}, \alpha x_{2}, \cdots ) .$ It is straightforward to check that $\alpha I$ is an $\mathcal{A}$-linear operator on $H_{\mathcal{A}}.$ Moreover, $\alpha I$ is bounded and $\parallel \alpha I \parallel = \parallel \alpha \parallel.$ Finally, $\alpha I$ is adjointable and its adjoint is given by $(\alpha I)^{*}=\alpha^{*}I.$ 

We introduce then the following notion:\\
$\sigma^{\mathcal{A}}(F)=\lbrace \alpha \in \mathcal{A} \mid F-\alpha I  $  is not invertible in$B^{a}(H_{\mathcal{A}}) \rbrace ;$\\
$\sigma_{p}^{\mathcal{A}}(F)=\lbrace \alpha \in \mathcal{A} \mid \ker(F-\alpha I) \neq \lbrace 0 \rbrace \rbrace;$\\
$\sigma_{rl}^{\mathcal{A}}(F)=\lbrace \alpha \in \mathcal{A} \mid F-\alpha I  $  is bounded below, but not surjective on $H_{\mathcal{A}}) \rbrace ;$\\
$\sigma_{cl}^{\mathcal{A}}(F)=\lbrace \alpha \in \mathcal{A} \mid Im(F-\alpha I)  $   is \underline{not} closed $\rbrace.$ (where $F \in B^{a}(H_{\mathcal{A}}) )$). 

Recall that not all closed submodules of $H_{\mathcal{A}}$ are orthogonally complementable in $H_{\mathcal{A}} $ which differs from the situation of Hilbert spaces. It may happen that $\overline{Im (F-\alpha I)} \oplus Im (F-\alpha I)^{\perp}  \subsetneqq H_{\mathcal{A}} .$  However, if $Im (F-\alpha I) $ is closed, then $Im(F^{*}-\alpha^{*}I) $ is closed and we also have $H_{\mathcal{A}}=Im(F-\alpha I) \oplus \ker (F^{*}-\alpha^{*}I)=\ker (F-\alpha I) \oplus Im (F^{*}-\alpha^{*}I)  $ whenever $F \in B^{a}(H_{\mathcal{A}}) .$

Therefore, it is more convinient in this setting to work with $\sigma_{rl}^{\mathcal{A}}(F) $ and $\sigma_{cl}^{\mathcal{A}}(F) $ for $F \in B^{a}(H_{\mathcal{A}}) $ instead of the residual and the continuous spectrum. 

	Note that we have obviously $\sigma^{\mathcal{A}}(F)=\sigma_{p}^{\mathcal{A}}(F) \cup \sigma_{rl}^{\mathcal{A}}(F) \cup \sigma_{cl}^{\mathcal{A}}(F) $ and $\sigma^{\mathcal{A}}(F^{*})=(\sigma^{\mathcal{A}}(F))^{*} .$

\begin{flushleft}
	The challenges which arise are the following:
\end{flushleft}
1) $\mathcal{A}$ may be non commutative; \\
2) Not all non-zero elements in $\mathcal{A}$ are invertible. Moreover, even if $ \alpha \in \mathcal{A} \cap G(\mathcal{A}) ,$ we do not have in general that $ \parallel \alpha^{-1} \parallel = \dfrac{1}{\parallel \alpha \parallel} .  $  Therefore $ \sigma^{\mathcal{A}}(F) $ may be unbounded. ( However $ \sigma^{\mathcal{A}}(F) $ is always closed in $\mathcal{A}$  ). 

The notion of generalized spectra has originally been introduced in \cite{IS4}. In this paper we give description of the generalized spectra of shift operators, unitary, self-adjoint and normal operators on $H_{\mathcal{A}}$ and investigate some further properties of these spectra. Most of the results in this paper are generalizations of the results from \cite{KU}.\\

\section{Main results}
We start with the following proposition.
\begin{proposition} \label{09p 01} 
	Let $\mathcal{A}$ be a unital $C^{*} $-algebra $\lbrace e_{k} \rbrace_{k \in \mathbb{N}} $ denote the standard orthonormal basis of $H_{\mathcal{A}} $ and $S$ be the operator defined by $Se_{k}=e_{k+1}, k \in \mathbb{N},$ that is $S$ is unilateral shift and $S^{*}e_{k+1}=e_{k} $ for all $k \in \mathbb{N}.$ If $\mathcal{A}=L^{\infty} ((0,1)) $ or if $\mathcal{A}=C([0,1]) ,$ then $\sigma^{\mathcal{A}}(S)=\lbrace \alpha \in \mathcal{A} \mid \inf \vert \alpha \vert \leq 1 \rbrace ,$  (where in the case when $\mathcal{A}=L^{\infty} ((0,1)),$ we set $\inf \vert \alpha \vert = \inf \lbrace C >0 \mid \mu ( \vert \alpha \vert^{-1}[0,C])>0 \rbrace = \sup \lbrace K>0 \mid \vert \alpha \vert > K ) \text{ a.e. on } (0,1)\rbrace ).$ Moreover, $\sigma_{p}^{\mathcal{A}}(S)= \varnothing $ in both cases. 
\end{proposition}
\begin{proof}
	We have two cases. Case 1: In this case we consider $\mathcal{A}=C([0,1]) .$ Let $\alpha \in \mathcal{A} $ and suppose that inf $ \vert \alpha \vert <1.$ Since $\vert \alpha \vert $ is continuous, we may find an open interval $(t_{1},t_{2})  \subseteq (0,1)  $ such that $\vert \alpha (t) \vert <1-\epsilon $ for all $t \in (t_{1},t_{2}) ,$ where $0 < \epsilon < 1-\inf \vert \alpha \vert .$ We may find some $ g \in \mathcal{A} $ sunce that supp $ g \subseteq (t_{1},t_{2})$ and $0 \leq g \leq 1 .$ Consider $ x_{\alpha}=(g, \overline{\alpha}g,\overline{\alpha}^{2}g, \cdots  ) .$ Then $\langle  (\alpha I-S) e_{k},x_{\alpha} \rangle =\overline{\alpha}^{k}g - \overline{\alpha}^{k}g=0  .$ Hence $x_{\alpha} \in Im (\alpha I-S)^{\perp} $ and $x_{\alpha } \neq 0  ,$ which gives that $\alpha \in \sigma^{\mathcal{A}}(S)  .$ Hence $ \lbrace \alpha \in \mathcal{A} \mid \inf \vert \alpha \vert <1 \rbrace  \subseteq  \sigma^{\mathcal{A}}(S) .$ Since $\sigma^{\mathcal{A}}(S) $ is closed in the norm topolgy in $\mathcal{A} $ it follows that inf $ \lbrace \alpha \in \mathcal{A} \mid \inf \vert \alpha \vert \leq 1 \rbrace  \subseteq  \sigma^{\mathcal{A}}(S) .$  On the other hand, if $\alpha \in \mathcal{A} $ and  $\inf \vert \alpha \vert > 1 ,$ then $\alpha $ is invertible and  $\sup \vert \alpha^{-1} \vert = \parallel \alpha^{-1} \parallel <1 .$ It follows that $\parallel \alpha^{-1}S \parallel \leq \parallel \alpha^{-1} \parallel \parallel S \parallel <1,$ so $\alpha I - S = \alpha (I- \alpha^{-1}S ) $ is invertible in $B^{a}(H_{\mathcal{A}}) .$	
	
	Next, suppose that $(\alpha I - S)(x)=0 $ for some $\alpha \in \mathcal{A}  $ and $x \in H_{\mathcal{A}} .$ This gives the following system of equations coordinatewise: 
	$\alpha x_{1}=0, \alpha x_{2}-x_{1}=0, \alpha x_{3}-x_{2}=0, \cdots .  $
	Since $\alpha x_{1}=0 ,$ we deduce that ${x_{1}}_{\mid \text{supp } \alpha}=0 .$ However, since $ \alpha x_{2}-x_{1}=0 ,$ it follows that ${x_{1}}_{\mid (\text{supp } \alpha)^{c}}=0  $ also. Hence $x_{1}=0 .$ But then $\alpha x_{2}=0 $ and $\alpha x_{3}-x_{2}=0 .$ Using the same argument we obtain that $x_{2}=0 .$ Proceding inductively we obtain that $x_{k}=0 $ for all $k,$ so $x=0 .$ Since $\alpha \in \mathcal{A}  $ was arbitrary chosen, we conclude that $\sigma_{p}^{\mathcal{A}}(S)=   \varnothing .$\\
	Case 2: In this case we consider $\mathcal{A}=L^{\infty}((0,1)) .$ Let $\alpha \in \mathcal{A} $ and assume that inf $\vert \alpha \vert <1 .$ This means that $\mu(\vert \alpha \vert^{-1}((0,1-\epsilon))>0     ) .$  Set $M_{\epsilon}=\vert \alpha \vert^{-1}((0,1-\epsilon)) , $ there exists on $ \epsilon \in (0,1-inf \vert \alpha \vert)$ such that $\chi_{M_{\epsilon}} \neq 0 .$ Letting $\chi_{M_{\epsilon}} $ play the role of the function $g$ in the previous proof, we deduce by the same arguments that $ \sigma^{\mathcal{A}}(S) = \lbrace \alpha \in \mathcal{A} \mid \inf \vert \alpha \vert \leq 1 \rbrace  .$ Next, assume that $(\alpha I - S ) (x)=0 $ for some $\alpha \in \mathcal{A} $ and $x \in H_{\mathcal{A}} .$ As in the previous proof we get the system of equations $ \alpha x_{1}=0, \alpha x_{2}-x_{1}=0, \alpha x_{3}-x_{2}=0, \cdots . $ The first equation gives that $x_{1}=0  $ a.e. on $\vert \alpha \vert^{-1} (0,\infty)$ whereas the second equation gives $ x_{1}=0 $ a.e. on $\alpha^{-1}(\lbrace 0 \rbrace) .$ Hence $x_{1}=0 .$ Proceding inductively as in the previous proof we get $x=0,$ hence $\sigma_{p}^{\mathcal{A}}(S)  $ is empty also in this case. 
\end{proof}
\begin{lemma}  \label{06l 18} 
	Let $T \in B(H) $ and suppose that $T$ is invertible. Then the equation $(TI-S)x=y $ has a solution in $H_{\mathcal{A}} $ for all $e_{k}, k \in \mathbb{N} $ if and only if the sequence $(T^{-1},T^{-2}, \cdots ,T^{-k}, \cdots) $ is in $H_{\mathcal{A}}  .$ 
\end{lemma}
\begin{proof}
	For $k=1 ,$ if $(TI-S)x=e_{1} $ we must have $TB_{1}=I $ where $x=(B_{1}, B_{2}, \cdots) .$ Hence $ B_{1}=T^{-1}.$ Next $TB_{2}-B_{1}=0 ,$ so $TB_{2}=B_{1}=T^{-1} $ which gives $B_{2}=T^{-2} .$ Proceeding inductively, we obtain that $B_{k}=T^{-k} $ for all $k.$ So the equation $(TI-S)x=e_{1} $ has a solution in $H_{\mathcal{A}}  $ if and only if the sequence $T^{-1},T^{-2}, \cdots $ belongs to $H_{\mathcal{A}} .$ 
	
	Now, if $T^{-1},T^{-2}, \cdots \in H_{\mathcal{A}}$ then the sequence $x^{(k)} $ in $H_{\mathcal{A}}  $ given by\\
	$x_{n}^{(k)}=\left\{\begin{matrix}
	0&  \text{ if }   n \in \lbrace 1, \cdots , k-1 \rbrace  \\ 
	T^{-(n-k+1)}&  \text{ for } n \in \lbrace k,k+1, \cdots \rbrace 
	\end{matrix}\right. $ \\ 
	is the solution of the equation $(TI-S)x=e_{k} $ for each $k \in \mathbb{N} .$
\end{proof}

For each $ T \in B(H)$ let $N_{T} $ denote the set of all right annihilators of $T.$ For each right invertible $T,$ let $B_{T} $ denote right inverse of $T$ and $W$ be the set of all right invertible $T \in B(H) $ such that there exists a non zero sequence $\lbrace Y_{n} \rbrace \subseteq N_{T} $ with the property that the sequence $\lbrace Y_{1},B_{T}Y_{1}+Y_{2}, B_{T}^{2}Y_{1}+B_{T}Y_{2}+Y_{3}, \cdots \rbrace$ belongs to $l_{2}(B(H)) .$

In addition, let $Z$ be the set of all right invertible $T$ such that there is \underline{no} sequence $\lbrace Y_{n} \rbrace $ in $N_{T} $ with the property that the sequence $\lbrace B_{T}+Y_{1},B_{T}^{2}+B_{T}Y_{1}+Y_{2}, B_{T}^{3}+B_{T}^{2}Y_{1}+B_{T}Y_{2}+Y_{3}, \cdots \rbrace $ belongs to $l_{2}(B(H)) .$ Then we have the following proposition.
\begin{proposition} \label{09p 08} 
	Let $\mathcal{A}=B(H) $ and $S$ be the unilateral shift on $ H_{\mathcal{A}}.$ Then $\sigma^{\mathcal{A}}(S)=Z \cup W \cup \lbrace T \in B(H) \vert T $ is not right invertible$\rbrace.$
\end{proposition}
\begin{proof}
	If $(TI-S)x=y $ should have a solution for each $e_{k} ,$ then, by the similar calculations as in the proof of Lemma \ref{06l 18} we must have that $T$ is right invertible and belongs to $B(H) \setminus Z .$ Now, suppose that $T$ is right invertible and that $(TS-I)y=0 $ for some $y=\lbrace  Y_{1},Y_{2}, \cdots \rbrace \in H_{\mathcal{A}} .$ We get coordinatewise the system of equations $TY_{1}=0,TY_{2}=Y_{1},TY_{3}=Y_{2}, \cdots .$ It is easily seen that $(TS-I)y=0 $ has a non trivial solution if and only if $T$ belongs to $W.$ The proposition follows.
\end{proof}
\begin{corollary} \label{09c 03} 
	Let $\mathcal{A}$ be a commutative unital $C^{*}$-algebra. Then $\sigma^{\mathcal{A}}(S)=\mathcal{A} \setminus G(\mathcal{A}) \cup \lbrace \alpha \in G(\mathcal{A}) \vert (\alpha^{-1}, \alpha^{-2}, \cdots , \alpha^{-k}, \cdots ) \notin H_{\mathcal{A}} \rbrace .$ 
\end{corollary}
\begin{proof}
	Since $\mathcal{A}$ is commutative, then the set of right invertible elemnts is $G(\mathcal{A}) .$ Hence we can apply the arguments from the proof of Lemma \ref{06l 18}.
\end{proof}
\begin{corollary} \label{09c 04} 
	Let $\mathcal{A}$ be a unital $C^{*}$-algebra. If $1_{\mathcal{A}} $ denotes the unit in $\mathcal{A} ,$ then $1_{\mathcal{A}} \in \sigma^{\mathcal{A}}(S) .$ 
\end{corollary}
\begin{proof}
	We have obviously that the sequence $(1_{\mathcal{A}} ,1_{\mathcal{A}}, 1_{\mathcal{A}}, \cdots )=(1_{\mathcal{A}}^{-1}, 1_{\mathcal{A}}^{-2}, 1_{\mathcal{A}}^{-3}, \cdots) $ is not an element of $H_{\mathcal{A}} .$ Then apply the arguments from the proof of Lemma \ref{06l 18}.
\end{proof} 
\begin{example} \label{09e 15} 
	We may also consider a weighted shift $S_{w} $ on $H_{\mathcal{A}} $ given by $S_{w}(x)_{j+1}=w_{j}x_{j} $ where $w=(w_{1},w_{2},\cdots)$ is a bounded sequence in $\mathcal{A} .$ In this case, if $\alpha $ has a common right annihilator as $w_{j} $ for some $j \in \mathbb{N} ,$ then the sequence having this right annihilator in its $j $-th coordinate and $0$ elsewhere belongs to the kernel of $\alpha I - S_{w} .$ Hence $\alpha \in \sigma^{\mathcal{A}}(S_{w}) $ in this case.
\end{example}
\begin{remark}
	Notice that Proposition \ref{09p 08} can be generalized to arbitrary unital $C^{*}$-algebras.
\end{remark} 
\begin{example} \label{09e 07} 
	 Let $\mathcal{A}=L^{\infty} ((0,1)).$  Set $\tilde{S}(f_{1},f_{2}, \cdots)=(f_{1}\chi_{(0,\frac{1}{2})},f_{2}\chi_{(0,\frac{1}{2})}+f_{1}\chi_{(\frac{1}{2},1)},f_{3}\chi_{(0,\frac{1}{2})}+f_{2}\chi_{(\frac{1}{2},1)}, \cdots).  $ Then $\tilde{S} $ has the matrix 
	 $\left\lbrack
	 \begin{array}{ll}
	 1 & 0 \\
	 0& S \\
	 \end{array}
	 \right \rbrack
	 $
	 w.r.t. the decomposition $(H_{\mathcal{A}} \cdot \chi_{(0,\frac{1}{2})}) \oplus (H_{\mathcal{A}} \cdot \chi_{(\frac{1}{2},1)} ) .$ It follows that $\sigma^{\mathcal{A}}(\tilde{S})=\lbrace \alpha \in \mathcal{A} \mid \inf \vert\alpha \vert \cdot  \chi_{(\frac{1}{2},1)} \leq 1 \rbrace $  $ \cup \lbrace \alpha \in \mathcal{A} \mid \mu (\lbrace t \mid t \in (0,\frac{1}{2}) \text{ and } \alpha (t)=1 \rbrace ) >0 \rbrace .$
\end{example}

\begin{proposition} \label{09p 02} 
	Let $\alpha \in \mathcal{A} .$ We have\\
	1. If $\alpha I-F$ is bounded below, and $ F \in B^{a}(H_{\mathcal{A}})$ then $\alpha \in \sigma_{rl}^{ \mathcal{A}}(F) $ if and only if $\alpha^{*} \in \sigma_{p}^{ \mathcal{A}}(F^{*}).$ \\
	2. If $F, D \in B^{a}(H_{\mathcal{A}})$ and $D=U^{*}FU $ for some unitary operator $U,$ then $\sigma^{ \mathcal{A}}(F)=\sigma^{ \mathcal{A}}(D),\sigma_{p}^{ \mathcal{A}}(F)=\sigma_{p}^{ \mathcal{A}}(D),\sigma_{cl}^{ \mathcal{A}}(F)=\sigma_{cl}^{ \mathcal{A}}(D) $ and $\sigma_{rl}^{ \mathcal{A}}(F)=\sigma_{rl}^{ \mathcal{A}}(D) .$
\end{proposition}
\begin{proof}
	1) Suppose first that $ F- \alpha I $ is bounded below and $\alpha \in \sigma_{rl}^{\mathcal{A}}(F) .$ Then $Im ( F- \alpha I) $ is closed and $Im ( F- \alpha I)^{\perp} \neq \lbrace 0 \rbrace . $ Since Im $ Im ( F- \alpha I)^{\perp} = \ker ( F^{*}- \alpha^{*} I),$ it follows that $ \alpha^{*} \in \sigma_{p}^{\mathcal{A}}(F^{*}).$ Conversely, suppose that $\alpha^{*} \in \sigma_{p}^{\mathcal{A}}(F^{*}) $ and $F- \alpha I $ is bounded below. Then, again Im $Im (F- \alpha I) $ is closed and moreover $Im (F- \alpha I)^{\perp}=\ker ( F^{*}- \alpha^{*} I) \neq \lbrace 0 \rbrace .$ It follows that $\alpha \in \sigma_{rl}^{\mathcal{A}}(F) .$
	It is straightforward to prove the statement in 2.
\end{proof}	
\begin{proposition}  \label{09p 03} 
	Let $ U \in B^{a}(H_{\mathcal{A}})  $ be unitary. Then $\sigma^{\mathcal{A}} (U) \subseteq \lbrace \alpha \in \mathcal{A} \mid \parallel \alpha \parallel \geq 1 \rbrace  $ and $\sigma^{\mathcal{A}} (U) \cap G(\mathcal{A})  \subseteq \lbrace \alpha \in G(\mathcal{A}) \mid \parallel \alpha^{-1} \parallel, \parallel \alpha \parallel \geq 1 \rbrace .$
\end{proposition}
\begin{proof}
	 We have $\alpha I - V=(\alpha V^{*}-I)V $ and $ \parallel V^{*} \parallel =  \parallel V \parallel = 1 .   $ 
\end{proof}
Consider again the orthonormal basis $\lbrace e_{k} \rbrace_{k \in \mathbb{N}} $ for $H_{\mathcal{A}} .$ We may enumerate this basis by indexes in $ \mathbb{Z}.$ Then we get orthonormal basis $\lbrace e_{j} \rbrace_{j \in \mathbb{Z}} $ for $H_{\mathcal{A}} $ and we can consider bilateral shift operator $V$ w.r.t. this basis i.e. $Ve_{k}=e_{k+1}$ all $k \in \mathbb{Z} ,$ which gives $V^{*}e_{k}=e_{k-1} $ for all $k \in \mathbb{Z} .$ 
\begin{proposition} \label{09p 04} 
	Let $V$ be bilateral shift operator. Then the following holds\\
	1)	If $\mathcal{A}= C([0,1]),$ then $ \sigma^{\mathcal{A}} (V)=\lbrace f \in \mathcal{A} \mid \vert f \vert ([0,1]) \cap \lbrace 1 \rbrace \neq \varnothing \rbrace$  \\
	2)	If $\mathcal{A} =L^{\infty} ([0,1]) ,$ then $ \sigma^{\mathcal{A}} (V)=\lbrace f \in \mathcal{A} \mid \mu (\vert f \vert^{-1} ((1-\epsilon,1+\epsilon))>0 \text{  } \forall \epsilon >0 \rbrace . $ In both cases $\sigma_{p}^{\mathcal{A}}(V)= \varnothing .$
\end{proposition}
\begin{proof}
\begin{flushleft}
		Case 1: 
\end{flushleft}
	In this case we consider $\mathcal{A}=C([0,1]) .$ Suppose that $\alpha \in \mathcal{A} $ and $\vert \alpha (\tilde{t})\vert =1 $ for some $\tilde{t} \in [0,1] .$ Choose a function $y \in \mathcal{A} $ such that $y(\tilde{t})=1 .$ If $\alpha I - V $ is surjective, then there exists an $x \in H_{\mathcal{A}} $ such that $(\alpha I - V )x=e_{1} \cdot y $ Now $x(\tilde{t}) \in l_{2} $ since $x \in H_{\mathcal{A}}  .$ It we let $\tilde{V} $ denote the ordinary bilateral shift on $l_{2} $ we get that $\alpha (\tilde{t})x(\tilde{t})-\tilde{V}(x(\tilde{t}))=(1,0,0, \cdots) ,$ as $y(\tilde{t})=1.$ But this is not possible since $\vert \alpha (\tilde{t}) \vert= 1 $ (for more details, see \cite[ Propozicija 19]{KU}. We conclude that $\alpha I -V $ can not be surjective, so $\alpha \in  \sigma^{\mathcal{A}}(V) .$ On the other hand, if $\alpha \in \mathcal{A} $ and $\vert \alpha \vert ([0,1]) \cap \lbrace 1 \rbrace =\varnothing  ,$ then either $\vert \alpha (t) \vert \geq C >1 $ or $\vert \alpha (t) \vert \leq K < 1 $ for all $ t \in [0,1]$ and some constants $C$ and $K$ (here we muse that $\vert \alpha \vert  $  is continuous, hence $\vert \alpha \vert ([0,1]) $ must be connected). If $\vert \alpha (t) \vert \geq C >1 $ for all $t \in [0,1] ,$ then $\alpha $ is invertible in $\mathcal{A}$ and $\parallel   \alpha^{-1} \parallel \leq \dfrac{1}{C} <1.$ Since $\parallel V \parallel=1 $ it follows that $\alpha \notin \sigma^{\mathcal{A}}(V)  $ then. If $\vert \alpha (t) \vert \leq K <1 $ $t \in [0,1] ,$ then $\parallel \alpha \parallel \leq K <1,$ so again using that $\alpha I-V= V(\alpha V^{*}-I)$ it follows that $\alpha \notin \sigma^{\mathcal{A}}(V)  $ then. Hence $ \sigma^{\mathcal{A}}(V) = \lbrace \alpha \in  \mathcal{A} \mid \vert \alpha \vert  ([0,1]) \cap \lbrace 1 \rbrace \neq \varnothing \rbrace  $
	Next, if $(\alpha I - V)(x)=0  $ for some $x \in H_{\mathcal{A}},$ then we must have $\alpha (t)x(t) -\tilde{V}x(t)=0$ for all $t \in [0,1].$ This means that $x(t)=0 $ for all $t \in [0,1] .$
\end{proof}
\begin{flushleft}
	Case 2:
\end{flushleft}
Let now $\mathcal{A}=L^{\infty}((0,1)) $ and $\alpha \in \mathcal{A} $ be such that $\mu (\vert \alpha \vert^{-1}((1-\epsilon,1+\epsilon)))>0 $ for all $\epsilon >0 .$ If $(\alpha I - V)x=e_{0} $ for some $x \in H_{\mathcal{A}} ,$ then we must have $\alpha x_{k}-x_{k-1}=0 $ for all $k \neq 0 $ and $\alpha x_{0}-x_{-1}=1 .$ For small $\epsilon >0 $ set $M_{\epsilon}=\vert \alpha \vert^{-1}((1-\epsilon,1+\epsilon)),$ $M_{\epsilon}^{-}=\vert \alpha \vert^{-1}((1-\epsilon,1)),M_{\epsilon}^{+}=\vert \alpha \vert^{-1}((1,1+\epsilon)), $ so $ M_{\epsilon}=M_{\epsilon}^{-} \cup M_{\epsilon}^{+}, $ $\chi_{M_{\epsilon}} \neq 0 . $ From the equations above we get $x_{k}= \alpha ^{-(k+1)}x_{-1} $ for all $k \leq -1 $ and $x_{k}=\alpha^{-k}x_{0} $ for all $k \geq 1 $ on any subset of $(0,1)$ on which $\alpha $ is bounded below, thus in particular on $M_{\epsilon}.$ It follows that $x_{k}\chi_{M_{\epsilon}^{+}}=0 $ for all $k \leq -1 $ and $x_{k}\chi_{M_{\epsilon}^{-}}=0 $ for all $k \geq 0 .$ Setting this into the second equation above, we get $\alpha x_{0} \chi_{M_{\epsilon}^{+}}-x_{1} \chi_{M_{\epsilon}^{-}}=\chi_{M_{\epsilon}} $ which gives $x_{0}=\alpha^{-1}\chi_{M_{\epsilon}^{+}} $ and $x_{1}=-\chi_{M_{\epsilon}^{-}} .$ Hence $x_{k}=\alpha^{-(k+1)} \chi_{M_{\epsilon}^{+}} $ for all $k \geq 0 $ and $x_{k}=- \alpha^{-(k+1)} \chi_{M_{\epsilon}^{-}} $ for all $k \leq -1.$ This gives $\vert x_{k}  \vert \geq (1+\epsilon)^{-(k+1)}\chi_{M_{\epsilon}^{+}}$ for all $k \geq 0 $ and $\vert x_{k}  \vert \geq (1-\epsilon)^{-(k+1)}\chi_{M_{\epsilon}^{-}} $ for all $k \leq -1 .$ Since this holds for all $\epsilon >0 $ and moreover, we have that either $\chi_{M_{\epsilon}^{-}} $ or $\chi_{M_{\epsilon}^{+}} $is non zero (because $\chi_{M_{\epsilon}} $ for all $\epsilon >0 $),  we get that the infinite sum $\sum_{k \in \mathbb{Z}} x_{k}^{*}x_{k} $ diverge in $\mathcal{A} ,$ so $x$ can not be an element of $H_{\mathcal{A}} .$ We conclude that $e_{0} \notin Im(\alpha I - V) ,$  so $\alpha \in \sigma^{\mathcal{A}}(V)  $ in this case. 

On the other hand, if $\mu (\vert \alpha \vert^{-1}((1-\epsilon,1+\epsilon)))=0 $ for some $\alpha \in \mathcal{A} $ and some $\epsilon >0, $ then $(0,1)=N_{\epsilon}^{-} \cup N_{\epsilon}^{+} $ where $N_{\epsilon}^{-}= \vert \alpha \vert^{-1}((0,1-\epsilon))$ and $N_{\epsilon}^{+}= \vert \alpha \vert^{-1}((1+\epsilon,\parallel \alpha \parallel)) .$ Since the submodules $H_{\mathcal{A}}\cdot \chi_{N_{\epsilon}^{-}}  $ and $H_{\mathcal{A}}\cdot \chi_{N_{\epsilon}^{+}}  $ clearly reduce the operator $V$ and the restrictions of $\alpha I-V $ on both these submodules are invertible, it follows that $\alpha I-V  $ is invertible, so $\alpha \notin \sigma^{\mathcal{A}}(V) .$
\begin{example} \label{09e 08} 
	Let $\lbrace \alpha_{1}, \alpha_{2}, \cdots \rbrace$ be a sequence in a unital $C^{*}$-algebra $\mathcal{A}$ s.t. each $\alpha_{k} $ is a unitary element of $\mathcal{A}.$ Then the operator $V$ defined by $V(x_{1},x_{2}, \cdots)=(\alpha_{1} x_{1}, \alpha_{2} x_{2} , \cdots)$ is a unitary operator on $H_{\mathcal{A}} .$ We have $\sigma^{\mathcal{A}} (V)=\lbrace \beta \in \mathcal{A} \mid \beta-\alpha_{k}  $ is not right invertible in $\mathcal{A} $ or that $(\beta -\alpha_{k}) \gamma_{k}=0$ for some $\gamma_{k} \in \mathcal{A} \rbrace  .$ If $\mathcal{A}=C([0,1]) \text{ or if } \mathcal{A}=L^{\infty}((0,1))$ and $J_{1},J_{2} $ are two closed subintervals of $(0,1) $ such that $J_{1} \cap J_{2}=\varnothing,$ then we may easily find a function $\beta \in \mathcal{A} $ such that $\beta = \alpha_{1} $ on $J_{1} $ and $\vert \beta (t) \vert >1 $ for all $t \in J_{2} .$ Hence $\parallel \beta \parallel >1 ,$ but we also have $ \beta \in \sigma^{\mathcal{A}}(V) .$ Similarly, if $\mathcal{A}=B(H) $ where $H$ is a Hilbert space, then we may easily find two closed suspaces $H_{1} $ and $H_{2} $ such that $H_{1} \perp H_{2} $ and $T \in B(H) $ is such that $T_{\mid_{H_{1}}} = {\alpha_{1}}_{\mid_{H_{1}}} $ and $\parallel T_{\mid_{H_{2}}} \parallel >1 .$ Hence again $T \in  \sigma^{\mathcal{A}}(V)$ and $\parallel T \parallel >1 .$ So, if $V$ is a unitary operator on $H_{\mathcal{A}} ,$ we do \underline{not} have in general that $ \sigma^{\mathcal{A}}(V) \subseteq  \lbrace \alpha \in \mathcal{A} \mid \parallel \alpha \parallel=1 \rbrace .$ 
\end{example}
\begin{lemma} \label{09l 01} 
	If $F$ is a self-adjoint operator on $H_{\mathcal{A}} ,$ then $\sigma_{p}^{\mathcal{A}}(F) $ is a self-adjoint subset of $\mathcal{A} ,$ that is $\alpha \in \sigma_{p}^{\mathcal{A}}(F) $ if and only if $\alpha^{*} \in \sigma_{p}^{\mathcal{A}}(F) $ in the case when $\mathcal{A} $ is a commutative $C^{*}$-algebra.
\end{lemma}
\begin{proof}
	 Since $F-\alpha I $ and $F-\alpha^{*} I=F^{*}-\alpha^{*} I$ mutually commute, we can deduce that $\parallel (F-\alpha I)x \parallel = \parallel (F-\alpha^{*} I) x \parallel$ for all $x \in H_{\mathcal{A}} .$ 
\end{proof}
\begin{example} \label{09e 09} 
	Let $\mathcal{A}=C([0,1]) $ or $\mathcal{A}=L^{\infty}((0,1)) .$ If $G$ is the operator on $H_{\mathcal{A}} $ given by $ G( f_{1},f_{2},\cdots )=(g_{1} f_{1},g_{2}f_{2},\cdots ),$ where $\lbrace g_{1},g_{2},\cdots \rbrace $ is a bounded sequence of real valued functions in $\mathcal{A} ,$ then $G$ is a self-adjoint operator. Suppose that there are two mutually disjoint, closed subintervals $J_{1} $ and $J_{2} $ of $(0,1)$ such that ${g_{1}}_{\mid_{J_{1}}} \neq 0$ and ${g_{1}}_{\mid_{J_{2}}} = 0 .$ Set $\tilde{g} = ig_{1} .$ Then, if we choose a function $f$ in $\mathcal{A}$ such that supp $f \subseteq J_{2} ,$ we get that $(\tilde{g}I-G)(f,0,0, \cdots)=0 .$ However $\tilde{g} \neq \overline{\tilde{g}} ,$ so we do \underline{not} have that $\sigma_{p}^{\mathcal{A}}(G) $ is included in the set of self-adjoint elements of $\mathcal{A}.$
\end{example}
\begin{example} \label{09e 10} 
	Let $\mathcal{A}=B(H)$ where $H$ is a Hilbert space and let $\lbrace e_{j} \rbrace_{j \in \mathbb{N}} $ be an orthonormal basis for $H.$ If $P$ denotes the orthogonal projection onto  $Span \lbrace e_{1} \rbrace ,$ then the operator $P \cdot I$ is a self-adjoint operator on $H_{\mathcal{A}} .$ Now, if $S$ is the unilateral shift operator on $H$ w.r.t. to the orthonormal basis $\lbrace e_{j}  \rbrace,$ then $S-P$ is injective whereas $S^{*}-P$ is \underline{not} injective because $(S^{*}-P)(e_{1}+e_{2})=0.$ It follows that $(S-P)\cdot I$ is an injective operator whereas $(S^{*}-P)\cdot I=((S-P)\cdot I)^{*} $ is \textbf{not} an injective operator. Hence, if $\mathcal{A}=B(H) $ where $H$ is a Hilbert space, we do not have in general that $ \sigma_{p}^{\mathcal{A}}(F)$ is a self-adjoint subset of $\mathcal{A}$ when $F=F^{*} .$
\end{example}

\begin{lemma}\label{09l 02} 
	Let $\mathcal{A}$ be a commutative $C^{*}$-algebra. If $F$ is a self-adjoint operator on $H_{\mathcal{A}} $ and $\alpha \in \mathcal{A} \setminus \sigma_{p}^{\mathcal{A}}(F)  ,$ then $\overline{R(F-\alpha I)}^{\perp}=\lbrace 0 \rbrace .$ Hence, if $\alpha \in \mathcal{A} \setminus \sigma_{p}^{\mathcal{A}}(F)  $ and in addition $F-\alpha I $ is bounded below, then $\alpha \in \mathcal{A} \setminus \sigma^{\mathcal{A}}(F) .$ 
\end{lemma}
\begin{proof}
	Suppose now that $ \alpha \in \mathcal{A} \setminus \sigma_{p}^{\mathcal{A}}(F) .$ If $y \in \overline{Im(F-\alpha I)}^{\perp}  ,$ then for all $ x \in H_{\mathcal{A}} $ we have $  \langle (F-\alpha I )x,y \rangle=0.$ This gives $ \langle x,(F-\alpha^{*} I )y \rangle=0. $ for all $ x \in H_{\mathcal{A}}  .$ It follows that $ (F-\alpha^{*} I )y=0 $ in this case. By the arguments above we obtain the $(F-\alpha I )y=0   .$ Since $\alpha \notin \sigma_{p}^{\mathcal{A}}(F)  $ by the choice of $ \alpha ,$ we get that $ y=0 .$ Thus $\overline{Im(F-\alpha I)}^{\perp}= \lbrace 0 \rbrace ,$ when
	$\alpha \in \mathcal{A} \setminus \sigma_{p}^{\mathcal{A}}(F).$ 	
	
	Suppose next that $\alpha \in \mathcal{A}   $ is such that $F-\alpha I  $ is bounded below. Then $\alpha \in \mathcal{A} \setminus \sigma_{p}^{\mathcal{A}}(F)  $ so from the previous arguments we deduce that $Im(F-\alpha I)^{\perp}= \lbrace 0 \rbrace.$ Moreover, since  $Im(F-\alpha I)  $ is then closed and $F-\alpha I  \in B^{a}(H_{\mathcal{A}}),$ from \cite[Theorem 2.3.3]{MT} it follows that $Im(F-\alpha I) $ is orthogonally complementablle in $H_{\mathcal{A}}  .$ But, since $Im(F-\alpha I)^{\perp}= \lbrace 0 \rbrace  ,$ we must have that $ Im(F-\alpha I)=H_{\mathcal{A}}  .$ Hence $ F-\alpha I $ is invertible in $B^{a}(H_{\mathcal{A}})  ,$ so $\alpha$ is in $ \mathcal{A}\setminus \sigma^{\mathcal{A}}(F)  .$
\end{proof}
\begin{corollary} \label{09c 01} 
	Let $\mathcal{A}$ be a unital commutative $C^{*}$-algebra and $F$ be a self-adjoint operator on $H_{\mathcal{A}} .$ If $\alpha \in \mathcal{A} $ and $\alpha - \alpha^{*} \in G(\mathcal{A}) ,$ then $F-\alpha I $ is invertible. In this case $\parallel (F-\alpha I)^{-1} \parallel \leq 2 \parallel (\alpha - \alpha^{*})^{-1} \parallel . $ 
\end{corollary}
\begin{proof}
	If $\alpha \in \mathcal{A}  $ and $ \alpha - \alpha^{*} \in G(\mathcal{A}), $ we have $ \langle Fx-\alpha Ix,x \rangle - \langle x, Fx-\alpha Ix \rangle = \langle x,x \rangle \alpha - \alpha^{*} \langle x,x \rangle = (\alpha - \alpha^{*}) \langle x,x \rangle$ (here we use that $\mathcal{A}  $ is commutative). From the triangle inequality and the Cauchy-Schwartz  inequality for the inner product we obtain 
	$ \Vert (\alpha - \alpha^{*}) \langle x,x \rangle \Vert \leq 2 \Vert x \Vert \Vert Fx-\alpha Ix \Vert  .$ Since $ (\alpha - \alpha^{*}) $ is invertible by assumption, we get that $ \Vert (\alpha - \alpha^{*})   \Vert  \langle x,x \rangle \geq \frac{1}{\Vert (\alpha - \alpha^{*})^{-1} \Vert} \Vert \langle x,x \rangle \Vert .$ It follows that $ \Vert (F-\alpha I)(x)  \Vert \geq \frac{1}{2\Vert (\alpha - \alpha^{*})^{-1} \Vert} \Vert \langle x,x \rangle \Vert $ for all $x \in H_{\mathcal{A}}  .$
\end{proof}

\begin{remark}
	Let $\mathcal{A}=C([0,1]) \text{ or } \mathcal{A}=L^{\infty}((0,1)) .$ As we have seen in Example \ref{09e 09}  the operator $\tilde{g}I-G $ is not invertible, whereas $\tilde{g} - \overline{\tilde{g}} =1ig_{1} \neq 0.$ Therefore, the requirement that $a-a^{*}$ is invertible is indeed necessary in Corollary \ref{09c 01}.
\end{remark}
\begin{example} \label{09e 14} 
	Let $\mathcal{A}=M_{2}(\mathbb{C})$ and $ T_{1},T_{2} \in \mathcal{A}$ be given as 
	$T_{1}=\left\lbrack
	\begin{array}{ll}
	2 & 1 \\
	1& 0 \\
	\end{array}
	\right \rbrack
	$ , 
	$T_{2}=\left\lbrack
	\begin{array}{ll}
	0 & i \\
	i& i \\
	\end{array}
	\right \rbrack
	.$  
	Then $T_{1}$ is self-adjoint and $T_{2}-T_{2}^{*}=2i $
	$\left\lbrack
	\begin{array}{ll}
	0 & 1 \\
	1& 1 \\
	\end{array}
	\right \rbrack
	,$  
	so $T_{2}-T_{2}^{*}$ is invertible.
	Now $T_{1}-T_{2}=$
	$\left\lbrack
	\begin{array}{ll}
	1 & 1-i \\
	1-i& -i \\
	\end{array}
	\right \rbrack
	,$ 
	so $det(T_{1}-T_{2})=0$  which gives $T_{1}-T_{2}$ is not invertible. Hence the operator $F:=T_{1} \cdot I $ is a self-adjoint operator on $H_{\mathcal{A}},$ but $F-T_{2}\cdot I=(T_{1}-T_{2}) \cdot I $ is \underline{not} invertible. This shows that the assumption that $\mathcal{A} $ is commutative in Corollary \ref{09c 01} is indeed necessary.
\end{example}
For a self-adjoint operator $F$ on $H_{\mathcal{A}}  ,$ set $M(F)=\sup \lbrace \parallel \langle Fx,x \rangle \parallel \mid \parallel x \parallel =1 \rbrace $ and $ m(F)=\inf \lbrace \parallel \langle Fx,x \rangle \parallel \mid \parallel x \parallel =1 \rbrace.$ We have the following corollary.
\begin{corollary} \label{09c 02} 
	If $\mathcal{A}=C([0,1]) $ and $F$ is a self-adjoint operator on $H_{\mathcal{A}} ,$ then $\sigma^{\mathcal{A}}(F) \subseteq \lbrace f \in \mathcal{A} \mid \vert f \vert (0,1) \cap [m,M] \neq \varnothing \rbrace .$ If $\mathcal{A}=L^{\infty}((0,1)) $ and $F$ is a self-adjoint operator on $H_{\mathcal{A}} $ then $\sigma^{\mathcal{A}}(F) \subseteq \lbrace f \in \mathcal{A} \mid \mu (\vert f \vert^{-1} [m-\epsilon,M+\epsilon]) > 0 \text{ for some } \epsilon= \epsilon(f) > 0  \rbrace .$ 
\end{corollary}
\begin{proof}
	Let $F$ be a self-adjoint operator on $H_{\mathcal{A}} $ where $\mathcal{A}=L^{\infty}((0,1)) $ and let $\alpha \in \mathcal{A} $ be such that there exists an $\epsilon=\epsilon (\alpha) $ with the property that $\mu(\vert f \vert^{-1} ((m- \epsilon , M+\epsilon ))= 0 .$ Then $(0,1)=M_{1}  \cup M_{2}$ where $M_{1} $ and $M_{2} $ are Borel measurable, mutually disjoint subsets of $(0,1)$ satisfying $\vert \alpha \vert \chi_{M_{1}} \geq M+\epsilon $ and $\vert \alpha \vert \chi_{M_{2}} \leq m-\epsilon $ a.e. Then, for all $x \in H_{\mathcal{A}} $ we have $\langle (F - \alpha I ) x,x \rangle = \langle (F - \alpha I ) x,x \rangle \cdot \chi_{M_{1}} + \langle (F - \alpha I ) x,x \rangle \cdot \chi_{M_{2}} $
	
	Now, we have 
	$$ \parallel   \langle (F - \alpha I ) x,x \rangle   \parallel \geq $$ 
	$$ \parallel   \langle (F - \alpha I ) x,x \rangle    \chi_{M_{1}}  \parallel \geq \parallel  \overline{ \alpha } \langle x,x \rangle \chi_{M_{1}} \parallel   - \parallel   \langle Fx \cdot \chi_{1}, x \cdot \chi_{M_{1}}  \rangle \parallel \geq$$ 
	$$\geq(M+\epsilon)\parallel   \langle x,x \rangle    \chi_{M_{1}}  \parallel - 
	\parallel   \langle F(x \cdot  \chi_{M_{1}}) , x \cdot  \chi_{M_{1}} \rangle  \parallel \geq $$
	$$\geq (M+\epsilon)\parallel   \langle x,x \rangle    \chi_{M_{1}}  \parallel - 
	M\parallel   \langle x \cdot  \chi_{M_{1}} , x \cdot \chi_{M_{1}} \rangle   \parallel =$$
	$$= (M+\epsilon)\parallel   \langle x,x \rangle    \chi_{M_{1}}  \parallel - 
	M\parallel   \chi_{M_{1}} \langle x , x  \rangle \chi_{M_{1}}    \parallel =\epsilon \parallel   \langle x , x  \rangle \chi_{M_{1}}    \parallel$$
	(where we have used that 
	$$ \parallel   \langle Fy , y  \rangle = \parallel y \parallel^{2}   \parallel \langle F(\dfrac{y}{\parallel y \parallel} ), \dfrac{y}{\parallel y \parallel} \rangle    \parallel \leq \parallel \langle y,y \rangle    \parallel M ).$$ 
	Next $$\parallel \langle (F - \alpha I)x,x \rangle    \parallel \geq  $$
	
	$$\geq \parallel \langle (F - \alpha I)x,x \rangle  \chi_{M_{2}}   \parallel \geq \parallel \langle Fx,x \rangle  \chi_{M_{2}}   \parallel - \parallel \tilde{\alpha} \langle x,x \rangle  \chi_{M_{2}}   \parallel = $$ 
	
	$$=\parallel \langle F (x  \chi_{M_{2}}  ),x \chi_{M_{2}}  \rangle \parallel - \parallel \tilde{\alpha} \langle x,x \rangle  \chi_{M_{2}}   \parallel \geq m \parallel \langle x \cdot \chi_{M_{2}}  ,x \cdot \chi_{M_{2}}  \rangle \parallel -$$ 
	$$-(m - \epsilon)\parallel \langle x  ,x \rangle \chi_{M_{2}}   \parallel = \epsilon \parallel \langle x  ,x \rangle \chi_{M_{2}}   \parallel.  $$
	Hence $\parallel \langle (F - \alpha I)x,x \rangle  \parallel \geq $
	$\epsilon \max  \lbrace \parallel \langle x  ,x \rangle \chi_{M_{2}}   \parallel ,  \parallel \langle x  ,x \rangle M_{2}   \rangle \parallel  = \epsilon  \parallel \langle x  ,x \rangle    \rangle \parallel  .$ 
	
	Thus $\parallel (F - \alpha I)x\parallel \parallel x \parallel  \geq
	\parallel \langle (F - \alpha I)x,x \rangle \parallel \geq \epsilon  \parallel x \parallel^{2} $ for a $x \in H_{\mathcal{A}} .$ It follows that $F - \alpha I $ is bounded below, hence from the previous result we deduce that $F - \alpha I $ is invertible in $B^{a}(H_{\mathcal{A}}) .$ 
	
	The proof in the case when $\mathcal{A}=C([0,1]) $ is similar, but more simple, because if $\alpha \in \mathcal{A} $ and $\vert \alpha \vert (0,1) \cap [m,M]= \varnothing ,$ then by continuity of $\vert \alpha \vert $ we must either have that $\vert \alpha \vert < m $ or $\vert \alpha \vert > M $ that  on the whole interval $[0,1] .$ Moreover, there exist then $\epsilon >0 $ such that $\vert \alpha \vert m-\epsilon $ or $\vert \alpha \vert \geq M+ \epsilon $ on the whole $[0,1] $ (since $\vert \alpha \vert $ reaches its maximum and minimun on $ [0,1] ).$ Then we may proceed in the same way as in the proof above.
\end{proof}
\begin{lemma} \label{09l 03} 
	Let $\mathcal{A}$ be a commutative unital $C^{*}$-algebra and $F$ be a normal operator on $H_{\mathcal{A}} ,$ that is $FF^{*}= F^{*}F.$ If $\alpha_{1}, \alpha_{2} \in \sigma_{p}^{\mathcal{A}}(F) $ and $\alpha_{1} - \alpha_{2} $ is invertible in $\mathcal{A},$ then $\ker (F-\alpha_{1} I) \perp  \ker (F-\alpha_{2} I).$ 
\end{lemma}
\begin{proof}
	Since $F$ commutes with $F^{*}$ and $\mathcal{A}$ is a commutative unital $C^{*}$-algebra, then $F-\alpha_{2}I $ and $F^{*}-\alpha_{2}^{*}I $ commute. Hence $\ker (F-\alpha_{2}I)=\ker (F^{*}-\alpha_{2}^{*}I) .$ For $x_{1} \in \ker (F-\alpha_{2}I) $ and $x_{2} \in \ker (F-\alpha_{2} I)=\ker (F^{*}-\alpha_{2}^{*}I)$ we get $ \langle x_{2} - x_{1} \rangle (\alpha_{1} - \alpha_{2})=\langle x_{2} - x_{1} \rangle \alpha_{1} - \alpha_{2} \langle x_{2} - x_{1} \rangle = \langle x_{2},Fx_{1} \rangle - \langle F^{*}x_{2}, x_{1} \rangle = 0  .$ Since $( \alpha_{1} - \alpha_{2} ) $ is invertible by assumption, it follows that $\langle x_{2},x_{1} \rangle =0 .$
\end{proof}
\begin{example} \label{09e 11} 
	Let  $\mathcal{A}=C([0,1]) $ or $\mathcal{A}=L^{\infty}((0,1)),$ consider the self-adjoint operator $G$ from the previous example. For any function $f$ in $\mathcal{A}$ with the support contained in $J_{2} ,$ we have $(f,0,0, \cdots) \in \ker G \cap \ker (\tilde{g}I-G).$ However $\tilde{g}=ig_{1}\neq 0 $ and $f  \neq 0,$ but $\tilde{g}  $ is \underline{not} invertible in $\mathcal{A}.$
\end{example}
\begin{example}
	Let $\mathcal{A}=B(H) $ and $T \in \mathcal{A} $ be a normal and invertible operator. If $H_{1} $ and $H_{2} $ are two closed subspaces of $H$ such that $H=H_{1} \tilde{ \oplus} H_{2} $ and $ H_{1} \neq H_{2}^{\perp} $ (that is $H_{1} $ and $H_{2} $ are not mutually orthogonal), then $T \Pi  $ and $T(1-\Pi) $ are elements of $\sigma_{p}^{\mathcal{A}}(T \cdot I) ,$ where $\Pi $ stands for the skew projection onto $H_{1} $ along $ H_{2}.$ Moreover, the operator $T \cdot I $ is normal operator on $H_{\mathcal{A}} $ and $ T \Pi  - T(1- \Pi )$ is invertible in $\mathcal{A} .$ However, if $ P_{1}$ and $P_{2} $ denote the orthogonal projections onto $H_{1} $ and $H_{2} ,$ respectively, then $e_{j}\cdot P_{1} \in \ker (T\Pi \cdot I- T\cdot I), e_{j} \cdot P_{2} \in \ker (T(I - \Pi )\cdot I - T \cdot I ) $ for all $j $ and $P_{1}P_{2} \neq 0 .$ So the assumption that  $\mathcal{A}$ is commutative is indeed necessary in Lemma  \ref{09l 03}.
\end{example} \label{09e 13} 
\begin{lemma} \label{09l 04} 
	Let $\mathcal{A}$ be a commutative $C^{*}$-algebra and $F$ be a normal operator on $H_{\mathcal{A}}.$ Then $ \sigma_{rl}^{\mathcal{A}}(F)=\varnothing ,$ hence $\sigma^{\mathcal{A}}(F)=\sigma_{p}^{\mathcal{A}}(F) \cup \sigma_{cl}^{\mathcal{A}}(F) .$ 
\end{lemma} 
\begin{proof}
	Suppose that $\alpha \in \sigma_{rl}^{\mathcal{A}}(F) .$ Then $ F-\alpha I $ is bounded below. Again, since $ F-\alpha I $ and $ F^{*}-\alpha^{*} I $ commute (because $\mathcal{A}$ is commutative, so $\alpha I$ commutes with $\alpha^{*}I ,$) we get that $\ker (F-\alpha I) =\ker (F^{*}-\alpha^{*} I)= \lbrace 0 \rbrace   .$ Next, since $Im (F-\alpha I) $ is closed, by \cite[Theorem 2.3.3]{MT},  we have $H_{\mathcal{A}}=\ker (F^{*}-\alpha^{*}I) \oplus Im(F-\alpha I)= Im (F-\alpha I) .$ So $F-\alpha I $ is surjective, thus invertible, which gives that $ \sigma_{rl}^{\mathcal{A}}(F) = \varnothing .$
\end{proof}
\begin{example}
	Let $\mathcal{A}=B(H) $ and $S, P$ be as in Example \ref{09e 10}. Then $P \cdot I $ is a normal operator on $H_{\mathcal{A}} $ being self-adjoint and $(S-P) \cdot I$ is bounded below on $ H_{\mathcal{A}}.$ Indeed, $\parallel (S-P)x \parallel \geq \parallel x \parallel  $ for all $x \in H ,$ hence $m(S-P) \geq 1 $ so $T^{*}(S-P)^{*}(S-P)T \geq m(S-P)^{2} T^{*}T $ for all $T \in B(H) $ which gives that $(S-P)\cdot I$ is bounded below on $ H_{\mathcal{A}} .$ However, $Im ((S-P)\cdot I)^{\perp}= \ker ((S^{*}-P)\cdot I) .$ and $\ker ((S^{*}-P) \cdot I) \neq \lbrace 0 \rbrace  $ as we have seen in Example \ref{09e 10}. Hence $ P \cdot I$ is a normal operator on $H_{\mathcal{A}} $ and $S \in \sigma_{rl}^{\mathcal{A}}(P \cdot I) $ which shows that the assumption that $\mathcal{A} $ is commutative is indeed necessary in Lemma \ref{09l 04}.
\end{example}
Next, for $F \in B^{a}(H_{\mathcal{A}}) ,$ set $\sigma_{a}^{\mathcal{A}}(F)=\lbrace \alpha \in \mathcal{A} \mid F-\alpha I $ is not bounded below $\rbrace.$
The next two propositions can be proved in exactly the same way same way as for operators on Hilbert spaces.
\begin{proposition} \label{09p 05} 
	For $F \in B^{a}(H_{\mathcal{A}}) ,$ we have that $\sigma_{a}^{\mathcal{A}}(F) $ is a closed subset of $\mathcal{A}$ in the norm topology and $\sigma^{\mathcal{A}}(F)=\sigma_{a}^{\mathcal{A}}(F) \cup \sigma_{rl}^{\mathcal{A}}(F) .$ 
\end{proposition}
\begin{proposition} \label{09p 07} 
	If $F \in B^{a}(H_{\mathcal{A}}) ,$ then $\delta \sigma^{\mathcal{A}}(F) \subseteq \sigma_{a}^{\mathcal{A}}(F).$ Moreover, if $M$ is a closed submodule of $H_{\mathcal{A}} $ and invariant with respect to $F,$ and $F_{0}=F_{\mid_{M}},$ then we have $\partial  \sigma^{\mathcal{A}}(F_{0}) \subseteq \sigma_{a}^{\mathcal{A}}(F) ,$ $\sigma^{\mathcal{A}}(F_{0}) \cap \sigma^{\mathcal{A}}(F) = \sigma_{rl}^{\mathcal{A}}(F_{0}).$
\end{proposition}
\begin{example} \label{09e 12} 
	We may also consider the operators on $H_{\mathcal{A}} $ when $\mathcal{A} $ is a unital $ C^{*}$-algebra defined as $W^{\prime}(e_{k}) =e_{2k}$ and $W^{\prime \prime}(e_{k}) =e_{2k-1}$ $ \forall k \in \mathbb{N} .$ Also for these operators we have $\sigma^{\mathcal{A}}(W^{\prime})= \sigma^{\mathcal{A}}(W^{\prime \prime})=\lbrace \alpha \in \mathcal{A} \mid \inf \vert \alpha \vert \leq 1 \rbrace $ in the case when $\mathcal{A} = C([0,1]) $ or when $\mathcal{A}=L^{\infty}((0,1)) .$ Suppose now that $\mathcal{A}=L^{\infty}((0,1)) $ and consider the operator $F$ on $H_{\mathcal{A}} $ given by $F(f_{1},f_{2},f_{3},\cdots) = ( \chi_{(0,\frac{1}{2})})f_{1} , \chi_{(\frac{1}{2},1)}f_{1},\chi_{(0,\frac{1}{2})})f_{2}, \chi_{(\frac{1}{2},1)}f_{2}, \cdots  ).$ It follows that $F$ has the matrix
	$\left\lbrack
	\begin{array}{ll}
	W^{\prime \prime} & 0 \\
	0& W^{\prime} \\
	\end{array}
	\right \rbrack
	$ 
	w.r.t the decomposition
	$(H_{\mathcal{A}} \cdot  \chi_{(0,\frac{1}{2})})  \oplus       (H_{\mathcal{A}} \cdot  \chi_{(\frac{1}{2},1)}) .$ Therefore $\sigma^{\mathcal{A}}(F)=\lbrace \alpha \in \mathcal{A} \mid \inf \vert \alpha \vert \leq 1 \rbrace .$ Next $\sigma_{p}^{\mathcal{A}}(W^{\prime})=\varnothing,$
	$ \sigma_{p}^{\mathcal{A}}(W^{\prime \prime})=\lbrace \alpha \in \mathcal{A} \mid \alpha=1 \text{ on some closed subinterval } J \subseteq [0,1]\rbrace $ in the case when $\mathcal{A} = C([0,1]) $ and $ \sigma_{p}^{\mathcal{A}}(W^{\prime \prime})=\lbrace \alpha \in \mathcal{A} \mid \mu (\lbrace t \in (0,1) \mid \alpha(t)=1 \rbrace ) >0 \rbrace $ in the case $\mathcal{A}=L^{\infty}((0,1)).$ Hence $\sigma_{p}^{\mathcal{A}}(F)=\lbrace \alpha \in \mathcal{A} \mid \mu (\lbrace t \in (0,\dfrac{1}{2}) \mid \alpha(t)=1 \rbrace ) >0 \rbrace .$

\end{example}
Consider now the operators\\
 $Z(e_{j})=\left\{\begin{matrix}
 e_{k}&  \text{ when }  j=2k  \\ 
 0&  \text{ else } 
 \end{matrix}\right. $ 
 , $k \in \mathbb{N} $\\
 $Z^{\prime}(e_{j})=\left\{\begin{matrix}
e_{k}&  \text{ when }  j=2k-1  \\ 
0& \text{ else } 
\end{matrix}\right. $ 
 , $k \in \mathbb{N} $  \\
Then $\sigma^{\mathcal{A}}(Z)=\sigma^{\mathcal{A}}(Z^{\prime})=\lbrace \alpha \in \mathcal{A} \mid \inf \vert \alpha \vert \leq 1 \rbrace . $  
This follows since $Z={W^{\prime}}^{*} and $   $Z^{\prime}={W^{\prime \prime }}^{*} .$ Moreover, we have $\sigma_{p}^{\mathcal{A}}(Z)= \lbrace \alpha \in \mathcal{A} \mid \inf \vert \alpha \vert < 1 \rbrace $ in both cases.  In the case $\mathcal{A} =L^{\infty}((0,1)) $ we have $\sigma_{p}^{\mathcal{A}}(Z^{\prime})= \lbrace \alpha \in \mathcal{A} \mid \inf \vert \alpha \vert < 1 \text{ or } \mu ( \lbrace t \in (0,1) \mid \alpha (t)=1) \rbrace>0 \rbrace  $ and in the case when $\mathcal{A} = C([0,1])$ we have $ \sigma_{p}^{\mathcal{A}}(Z^{\prime}) = \lbrace \alpha \in \mathcal{A} \mid \inf \vert \alpha \vert < 1 \text{ or } \alpha=1 \text{ on some closed subinterval } J \subseteq [0,1] \rbrace .$ Let operator $D$ on $H_{\mathcal{A}} $ be given by  $D(g_{1},g_{2},g_{3},\cdots) = (g_{1} \chi_{(0,\frac{1}{2})}+ g_{2}\chi_{(\frac{1}{2},1)},g_{3}\chi_{(0,\frac{1}{2})})+ g_{4}\chi_{(\frac{1}{2},1)}, \cdots  )$ when $\mathcal{A}=L^{\infty}((0,1))   .$ Then $D= F^{*}$ and $D$ has the matrix 
$\left\lbrack
\begin{array}{ll}
Z^{\prime } & 0 \\
0& Z \\
\end{array}
\right \rbrack
$
 w.r.t. the decomposition $H_{\mathcal{A}} \chi_{(0,\frac{1}{2})} \oplus H_{\mathcal{A}} \chi_{(\frac{1}{2},1)}$ It follows that $\sigma^{\mathcal{A}}(D)= \lbrace \alpha \in \mathcal{A} \mid \inf \vert \alpha \vert \leq 1 \rbrace $ and 
 $\sigma_{p}^{\mathcal{A}}(D)= \lbrace \alpha \in \mathcal{A} \mid \inf \vert \alpha \vert < 1 \text{ or } \mu ( \lbrace t \in (0,\frac{1}{2}) \mid \alpha (t)=1) \rbrace >0 \rbrace    . $\\
 
Now we are going to prove the statements in this example. Again, we consider two cases.\\
	Case 1: $ \mathcal{A} =L^{\infty}(0,1).$ If $\inf \vert \alpha \vert <1 $ for $\alpha \in \mathcal{A} ,$ then there exist an $\epsilon >0 $ such that $\mu (\vert \alpha \vert^{-1} ((0, 1-\epsilon))) >0 .$ Set $M_{\epsilon}=\vert \alpha \vert^{-1} ((0, 1-\epsilon)) $ and put $x_{\alpha}=\sum_{K=0}^{\infty} e_{2^{k}} \cdot \chi_{M_{\epsilon}} \overline{ \alpha }^k .$ Then it is easily seen that $x_{\alpha} $ indeed is an element of $H_{\mathcal{A}} $ and moreover $\langle (\alpha I - W^{\prime})e_{k}, x_{\alpha I} \rangle =0 $ for all $k, $ so $x_{\alpha} \in Im(\alpha I - W^{\prime})^{\perp}  .$ Therefore, $ \alpha \in \sigma^{\mathcal{A}} (W^{\prime}) .$ On the other hand, we have $\parallel W^{\prime} \parallel =1 ,$ hence if $inf \vert \alpha \vert >1,$ it follows by the same arguments as before that $\alpha I - W^{\prime} = \alpha (I-\alpha^{-1}) W^{\prime} $ is  invertible in $B^{a}(H_{\mathcal{A}}) .$ Hence, using that $\sigma^{\mathcal{A}} (W^{\prime}) $ is closed in the norm topology of $\mathcal{A},$ we obtain that $\sigma^{\mathcal{A}} (W^{\prime})= \lbrace \alpha \in \mathcal{A} \mid inf \vert \alpha \vert \leq 1 \rbrace.$ Next, if $(\alpha I - W^{\prime})(x)=0 ,$ then we must have $\alpha x_{2k}-x_{k} =0$ for all $k \in \mathbb{N} $ which gives $x_{k}=0 $ a.e.  on $\lbrace t \in (0,1) \mid \alpha (t)=0 \rbrace $ for all $ k \in \mathbb{N}.$ Moreover, $ \alpha x_{m}=0$ for all odd $m.$ All this together implies that $x_{m} $ for all odd $m.$ Now, if $j=2^{k}m $ where m is odd, we get $x_{m}=\alpha^{k}x_{j} .$ Since $x_{m}=0 ,$ it follows that $x_{j}=0 $ a.e. on $ \lbrace t \in (0,1) \mid \alpha (t) \neq 0 \rbrace .$ But since $x_{n}=0 $ a.e. on $ \lbrace t \in (0,1) \mid \alpha (t) = 0 \rbrace .$ for all $n \in \mathbb{N} ,$ we must have $x_{j}=0 .$ In particular $ x_{2}=0 .$ Then, from the arguments above it also follows that $x_{2^{k}}=x_{2^{k-1} \cdot 2}=0 $ for all $k.$ Thus $x=0 ,$ so $\sigma^{\mathcal{A}}(W^{\prime}) =  \varnothing .$ The proof for the case when $\mathcal{A}=C([0,1]) $ is similar.

	Next, we consider the operator $W^{\prime \prime}$. If inf $\vert \alpha \vert <1  $ for some $\alpha \in \mathcal{A}= L^{\infty}((0,1))  ,$ we let $ M_{\epsilon} \subseteq (0,1) $ be the same set as before. Put $y_{\alpha}=\sum e_{r_{k}} \chi_{M_{\epsilon}} \overline{ \alpha }^{(k-1)} ,$ where $\lbrace r_{k} \rbrace $ is the strictly increasing sequence of positive integers recursively defined by $r_{k+1}=2r_{k}-1 $ for $k \in \mathbb{N} $ and $r_{1}=2 .$ Then we have $\langle  (\alpha I -W^{\prime \prime})e_{k},y_{\alpha} \rangle $ for all $k \in \mathbb{N} ,$ so $y_{\alpha} \in Im (\alpha I - W^{\prime \prime })^{\perp}  .$ By the same arguments as before using that $\parallel W^{\prime \prime } \parallel=1 ,$  we obtain that $\sigma^{\mathcal{A}}(W^{\prime \prime })= \lbrace \alpha \in \mathcal{A} \mid \inf \vert \alpha \vert \leq 1 \rbrace .$ 
	Again, the proof for the case when $\mathcal{A}=C([0,1]) $ is similar. Suppose next that $(\alpha I - W^{\prime \prime })x=0 $ for some $x \in H_{\mathcal{A}} .$ Then for all $k \in \mathbb{N} $ we must have $\alpha x_{2k-1} - x_{k} = 0 $ which gives $x_{k}=0 $ a.e. on $ \lbrace t \in (0,1) \mid \alpha (t)=0 \rbrace$ for all $k \in \mathbb{N}.$ Moreover $\alpha x_{2k} = 0 $ for all $k$ and all this together implies that $x_{2k}=0 $ for all $k.$ If $\alpha \neq 1 $ a.e. on $(0,1),$ then the equation $\alpha x_{1} - x_{1}=0 $ gives $x_{1}=0 .$ Assume now that $x_{2j-1} =0$ for $j \in \lbrace 1, \cdots m \rbrace .$ We have $  \alpha x_{2m+1} - x_{m+1}= 0 .$ If $m+1 $ is even, then we already have that $ x_{m+1}=0.$ If $ m+1$ is odd, then by induction hypothesis we must have $x_{m+1}=0 $ since $m+1 \leq 2m-1 $ in this case (we assume that $m \in \mathbb{N},$ so if $ m+1 $ is odd, then $m \geq 2 ).$ Hence $\alpha x_{2m+1} = 0 .$ Combining this together with the fact that $x_{2m+1}=0 $ a.e. on $\lbrace t \in (0,1) \mid \alpha (t)=0 \rbrace ,$ we deduce that $x_{2m+1}=0 .$ By induction we obtain that $x_{2k-1} = 0 $ for all $ k \in \mathbb{N}.$ Hence $x=0 .$ So, if $\alpha \in  \mathcal{A}$ is such that $ \alpha \neq 1 $ a.e. then $\alpha \notin \sigma_{p}^{\mathcal{A}}(W^{\prime \prime}) .$ On the other hand, if $\alpha \in \mathcal{A} $ is such that $\mu( \lbrace t \in (0,1) \mid \alpha (t) = 1 \rbrace ) > 0 ,$ then if we let $M=\alpha^{-1}(\lbrace 1 \rbrace) ,$ we have $\chi_{M} \neq 0 .$ Hence $(\alpha I - W^{\prime \prime}) (\chi_{M},0,0,0,\cdots)=0$ so $\alpha \in \sigma_{p}^{\mathcal{A}}(W^{\prime \prime})  $ in this case. The proof for the case when $\mathcal{A}=C([0,1]) $ is similar. 
	
	Indeed, the equation $\alpha x_{1}-x_{1} $ gives that $x_{1}=0 $ on $\text{supp }(\alpha - 1) .$ If $\text{supp } (\alpha - 1) \neq [0,1] ,$ then there exists an open interval $J_{0} \subseteq \text{supp } (\alpha - 1)^{c}  ,$  since $\text{supp } (\alpha - 1)^{c} $ is open. On $J_{0} $ we have that $\alpha - 1=0 ,$ hence by continuity of $\alpha - 1 ,$ we must have that $\alpha - 1=0 $ on $\overline{J_{0}} ,$ which is a closed interval. Hence, if there is no closed subinterval of $[0,1] $ on which $\alpha =1 ,$ we must have that $\text{supp } (\alpha - 1)=[0,1],$ which gives $x_{1} =0 .$ Then we can proceed in the same way as above. On the other hand, if $\alpha=1 $ on some subinterval $[t_{1},t_{2}] \subseteq [0,1]  ,$ then we may choose an $ x \in C([0,1]) $ such that $\text{supp } x \subseteq [t_{1},t_{2}] .$ Then, $e_{1} \cdot x \in \ker (\alpha I - W^{\prime \prime}) .$
	
	Now we consider the operators $Z$ and $ Z^{\prime}.$ Suppose that $\alpha \in \mathcal{A} = L^{\infty}((0,1)) $ and $\inf \vert \alpha \vert < 1 .$

	Let $\epsilon >0 $ be s.t. $\chi_{M_{\epsilon}} \neq 0 ,$ where $M_{\epsilon} =  \vert \alpha \vert^{-1} ((0,1- \epsilon))   .$ Set $ x_{\alpha}= \sum_{k=1}^{\infty} e_{2^{k}} \cdot \chi_{M_{\epsilon}} \alpha^{k-1},$ then $ x_{\alpha} \in H_{\mathcal{A}} , x_{\alpha} \neq 0$ and $(\alpha I - Z)x_{\alpha}=0 .$ Hence $\alpha \in \sigma_{p}^{\mathcal{A}}(Z) .$ Let now $x \in H_{\mathcal{A}} $ be such that  $\inf \vert \alpha \vert = 1 .$ Then, if $(\alpha I-Z)x=0 $ for some $x \in H_{\mathcal{A}} ,$ we must have $\alpha x_{k} -x_{2k}=0 $ for all $k \in \mathbb{N} .$ If $x_{j} \neq 0 $ for some $j \in \mathbb{N} ,$ then $ x_{2^{k}j}$ for all $k \in \mathbb{N} .$ Hence $\vert x_{2^{k}j} \vert \geq \vert x_{j} \vert  $ for all $k \in \mathbb{N} $ (since inf $\vert \alpha \vert = 1 $). This is improssible since $x_{j} \neq 0 ,$ so we deduce that $ x_{j} = 0 $ for all $ j \in \mathbb{N} ,$ hence $ x=0 .$ Therefore $\sigma_{p}^{\mathcal{A}}(Z) = \lbrace \alpha \in \mathcal{A} \mid \inf \vert \alpha \vert < 1 \rbrace  .$  Next, we consider the operator $Z^{\prime} .$ Suppose that $\alpha \in \mathcal{A} $ and  $ \inf \parallel \alpha \parallel <1   .$ Let $M_{\epsilon} $ and $\chi_{M_{\epsilon}} $ be as before. Set $x_{\alpha}= \sum_{k=1}^{\infty} e_{r_{k}} \cdot \chi_{M_{\epsilon}} \alpha^{k-1}, $ where $\lbrace r_{k} \rbrace $ is the strictly increasing sequence of positive integers recursively defined by $r_{1}=2, r_{k+1}=2r_{k}-1 .$ Then $ (\alpha I - Z^{\prime})x_{\alpha}=0.$ Hence $\alpha \in \sigma_{p}^{\mathcal{A}}(Z^{\prime}) .$ Assume now that $\alpha \in \mathcal{A} $ and $\mu( \lbrace t \in (0,1) \mid \alpha (t) = 1 \rbrace ) > 0 .$ Set $M=\alpha^{-1}(\lbrace 1 \rbrace ) ,$ then $\chi_{M} \neq 0  $ and $(\alpha I - Z^{\prime }) (\chi_{M},0,0,0,\cdots)=0,$ so $\alpha \in \sigma_{p}^{\mathcal{A}}(Z^{\prime })  $ in this case. On the other hand, if $\alpha \in \mathcal{A} $ is such that $\alpha \neq 1 $ a.e on $(0,1)$ then $x_{1}=0 $ for all $x \in H_{\mathcal{A}} $ satisfying $(\alpha I - Z^{\prime })x=0 .$ This is because $\alpha x_{1}-x_{1}=0 $ then. Assume in addition that $\inf \vert \alpha \vert =1 .$ If $(\alpha I - Z^{\prime })x=0  $ for some $x \in H_{\mathcal{A}} $ and $x_{j}\neq 0 $ for some $ j \geq 2 $ then we must have $x_{\tilde{r}_{k}}=\alpha^{k-1}x_{j} $ where $\lbrace \tilde{r}_{k} \rbrace $ is the strictly increasing sequence of positive integers recursively defined by $\tilde{r}_{1}=j, \tilde{r}_{k+1}=2\tilde{r}_{k}-1 .$ This gives $\vert x\tilde{r}_{k} \vert \geq x_{j}$ for all $k,$ which is improssible since $x \in H_{\mathcal{A}} .$ Hence $x=0 ,$ so we deduce that $\alpha \notin \sigma_{p}^{\mathcal{A}}(Z^{\prime}) $ if $\inf \vert \alpha \vert  \geq 1 $ and $\alpha \neq 1 $ a.e. on $(0,1).$ The proof for the case when $\mathcal{A} =C([0,1])$ is similar.

Indeed, if $\inf \vert \alpha \vert <1 ,$ then we can find a closed subinterval $[t_{1},t_{2}] \subseteq [0,1] $ such that $\vert \alpha (t) \vert  \leq 1-\epsilon$ for all $t \in  [t_{1},t_{2}] ,$ where $\epsilon \in (0,1-\inf \vert \alpha \vert) .$ Let $g \in \mathcal{A} $ be such that  $\text{supp } g \subseteq [t_{1},t_{2}] .$ Letting $g$ play the role of $\chi_{M_{\epsilon}} $ in the arguments above, we can show that $\lbrace \alpha \in \mathcal{A} \vert \inf \vert \alpha \vert <1 \rbrace \subseteq \sigma_{p}^{\mathcal{A}}(Z^{\prime}) .$ 

Next, the equation $\alpha x_{1}-x_{1} =0$ has a nontrivial solution in $\mathcal{A}=C([0,1]) $ if and only if $\text{supp } (\alpha - 1)\neq [0,1] $ which holds if and only if $\alpha = 1 $ on some closed subinterval $J \subseteq [0,1] .$ In this case we choose an $h \in \mathcal{A} $ such that  $\text{supp } h \subseteq J ,$ then $(h,0,0,0,\dots) \in \ker (\alpha I - Z^{\prime}) .$
\vspace{5pt}
\begin{center}
	Acknowledgement	
\end{center}
First of all I am very grateful to Professor Camillo Trapani and Professor Dragan Djordjevic for suggesting me to consider generalized spectra in $C^{*}$-algebras of operators on Hilbert $C^{*}$-modules.

Also, I am very grateful to Professor Ajit Iqbal Singh for inspiring comments, remarks and advises that led to the improved version of the paper.

\vspace{35pt}
\begin{flushleft}
	Stefan Ivkovi\'{c} \\
The Mathematical Institute of the Serbian Academy of Sciences and Arts, \\
p.p. 367, Kneza Mihaila 36, 11000 Beograd, Serbia,\\
Tel.: +381-69-774237 \\
email: stefan.iv10@outlook.com  

\end{flushleft}


\newpage
\begin{thebibliography}{99}


















\bibitem{IS4} [IS4] S. Ivkovi\'{c}, \textit{On compressions and generalized spectra of operators over C*-algebras}, Annals of Functional Analysis, http://link.springer.com/article/10.1007/s43034-019-00034-z



\bibitem{KU} [KU] S. Kurepa, Funkcionalna analiza, Elementi teorije operatora, Skolska knjiga, Zagreb, 1981






\bibitem{MT}[MT] V. M. Manuilov, E. V. Troitsky, \textit{Hilbert C*-modules}, In: Translations of Mathematical Monographs. 226, American Mathematical Society, Providence, RI, 2005.
















	
	
\end{thebibliography}
\end{document}